\documentclass[12pt]{amsart}
\usepackage{a4wide}
\usepackage{graphicx,eepic, bm, times}
\usepackage{amsthm,amsmath,amsfonts,amssymb}
\usepackage{multirow}

\newtheorem{lemma}{Lemma}[section]

\newtheorem{prop}[lemma]{Proposition}
\newtheorem{thm}[lemma]{Theorem}
\newtheorem{cor}[lemma]{Corollary}
\theoremstyle{definition}

\newtheorem{example}[lemma]{Example}

\theoremstyle{remark}
\newtheorem{remark}[lemma]{Remark}

\numberwithin{equation}{section} \numberwithin{table}{section}

\begin{document}

\title[Universal and periodic $\beta$-expansions]{On universal and periodic $\beta$-expansions, and the Hausdorff dimension of the set of all expansions}
\author{Simon Baker}
\address{School of Mathematics, The University of Manchester,
Oxford Road, Manchester M13 9PL, United Kingdom. E-mail:
simon.baker@manchester.ac.uk}

\date{\today}
 \begin{abstract}
In this paper we study the topology of a set naturally arising from the study of $\beta$-expansions. After proving several elementary results for this set we study the case when our base is Pisot. In this case we give necessary and sufficient conditions for this set to be finite. This finiteness property will allow us to generalise a theorem due to Schmidt and will provide the motivation for sufficient conditions under which the growth rate and Hausdorff dimension of the set of $\beta$-expansions are equal and explicitly calculable.
\end{abstract}
\maketitle
\section{Introduction}
Let $m\in\mathbb{N},$ $\beta\in(1,m+1]$ and $I_{\beta,m}=[0,\frac{m}{\beta-1}].$ We call a sequence $(\epsilon_{i})_{i=1}^{\infty}\in\{0,\ldots,m\}^{\mathbb{N}}$ a \textit{$\beta$-expansion} for $x$ if $$\sum_{i=1}^{\infty}\frac{\epsilon_{i}}{\beta^{i}}=x.$$ It is a simple exercise to show that $x$ has a $\beta$-expansion if and only if $x\in I_{\beta,m}.$ For $x\in I_{\beta,m}$ we denote the set of $\beta$-expansions for $x$ by $\Sigma_{\beta,m}(x)$, i.e., $$\Sigma_{\beta,m}(x)=\Big\{(\epsilon_{i})_{i=1}^{\infty}\in \{0,\ldots,m\}^{\mathbb{N}} : \sum_{i=1}^{\infty}\frac{\epsilon_{i}}{\beta^{i}}=x\Big\}.$$ It is a well known property that for $\beta\in(1,m+1)$ a point $x\in(0,\frac{m}{\beta-1})$ will typically have a non-unique $\beta$-expansion, see \cite{Baker,Erdos,GlenSid,KoLiDe,Sidorov1,Sidorov4}. The following set was introduced in \cite{FengSid} 
\begin{align*}
\mathcal{E}_{\beta,m,n}(x)=\Big\{&(\epsilon_{1},\ldots,\epsilon_{n})\in\{0,\ldots,m\}^{n}|\exists (\epsilon_{n+1}, \epsilon_{n+2}, \ldots)\in \{0,\ldots, m\}^{\mathbb{N}}\\
&:\sum_{i=1}^{\infty}\frac{\epsilon_{i}}{\beta^{i}}=x\Big\},
\end{align*} 
we refer to an element of $\mathcal{E}_{\beta,m,n}(x)$ as an \textit{$n$-prefix} for $x$. In what follows we fix the map $T_{\beta,i}(x)=\beta x-i$ for $i\in\{0,\ldots,m\}.$ Moreover, we let
$$\Omega_{\beta,m}(x)=\Big\{(a_{i})_{i=1}^{\infty}\in \{T_{\beta,0},\ldots, T_{\beta,m}\}^{\mathbb{N}}:(a_{n}\circ a_{n-1}\circ \ldots \circ a_{1})(x)\in I_{\beta,m}
 \textrm{ for all } n\in\mathbb{N}\Big\}$$
\noindent and
$$\Omega_{\beta,m,n}(x)=\Big\{(a_{i})_{i=1}^{n}\in \{T_{\beta,0},\ldots,T_{\beta,m}\}^{n}:(a_{n}\circ a_{n-1}\circ \ldots \circ a_{1})(x)\in I_{\beta,m} \Big\},$$for each $n\in\mathbb{N}$. For our purposes it is also useful to define $\Omega_{\beta,m,0}(x)$ to be the set consisting of the identity map. Typically we will denote an element of $\Omega_{\beta,m,n}(x)$ or any finite sequence of maps by $a$. When we want to emphasise the length of $a$ we will use the notation $a^{(n)}$. We also adopt the notation $a^{(n)}(x)$ to mean $(a_{n}\circ a_{n-1}\circ \ldots \circ a_{1})(x).$ The following technical lemma will be useful.
\begin{lemma}
\label{Bijection lemma}
\begin{enumerate}
	\item $Card(\mathcal{E}_{\beta,m,n}(x))=Card(\Omega_{\beta,m,n}(x)),$ where our bijection identifies $(\epsilon_{i})_{i=1}^{n}$ with $(T_{\beta,\epsilon_{i}})_{i=1}^{n}.$
	\item $Card(\Sigma_{\beta,m}(x))=Card(\Omega_{\beta,m}(x)),$ where our bijection identifies $(\epsilon_{i})_{i=1}^{\infty}$ with $(T_{\beta,\epsilon_{i}})_{i=1}^{\infty}.$
	\item A finite block $(\epsilon_{1},\ldots,\epsilon_{n})$ of elements from $\{0,\ldots,m\}$ appears in a $\beta$-expansion for $x$ if and only if there exists a finite sequence of maps $a,$ such that $(T_{\beta,\epsilon_{n}}\circ\ldots\circ T_{\beta,\epsilon_{1}}\circ a)(x)\in I_{\beta,m}.$
\end{enumerate}
\end{lemma}
\begin{proof}
The proofs of statements $1$ and $2$ are contained in \cite{Baker}. To prove statement $3$ we replicate an argument given in \cite{FengSid}. Suppose $(\epsilon_{1},\ldots,\epsilon_{n})$ appears in a $\beta$-expansion for $x$, then there exists $N\in\mathbb{N}$ and $(\delta_{1},\ldots,\delta_{N})\in \{0,\ldots,m\}^{N}$
such that $$x-\sum_{i=1}^{N}\frac{\delta_{i}}{\beta^{i}}-\sum_{i=1}^{n}\frac{\epsilon_{i}}{\beta^{N+i}}\in\Big[0,\frac{m}{\beta^{N+n}(\beta-1)}\Big].$$ A simple manipulation yields that this is equivalent to $$\beta^{N+n}x-\sum_{i=1}^{N}\delta_{i}\beta^{N+n-i}-\sum_{i=1}^{n}\epsilon_{i}\beta^{n-i}\in I_{\beta,m}.$$ However $$\beta^{N+n}x-\sum_{i=1}^{N}\delta_{i}\beta^{N+n-i}-\sum_{i=1}^{n}\epsilon_{i}\beta^{n-i}= (T_{\beta,\epsilon_{n}}\circ\ldots \circ T_{\beta,\epsilon_{1}}\circ T_{\beta,\delta_{N}}\circ \ldots \circ T_{\beta,\delta_{1}})(x).$$ Our result follows immediately.

\end{proof}

With Lemma \ref{Bijection lemma} in mind we also refer to an element of $\Omega_{\beta,m,n}(x)$ as an $n$-prefix for $x$. Naturally arising from Lemma \ref{Bijection lemma} are the sets $$S_{\beta,m,n}(x)=\Big\{a(x): a\in \Omega_{\beta,m,n}(x)\Big\}$$ and $$S_{\beta,m}(x)=\bigcup_{n=0}^{\infty} S_{\beta,m,n}(x).$$ After proving several elementary results for $S_{\beta,m}(x)$ we will study the case when $\beta$ is Pisot. Our main result will be the following. 
\begin{thm}
\label{First theorem}
Let $\beta$ be a Pisot number, then $S_{\beta,m}(x)$ is finite if and only if $x\in \mathbb{Q}(\beta).$
\end{thm} Recall that a Pisot number is a real algebraic integer greater than $1$ whose other Galois conjugates are of modulus strictly less than $1$. Using Theorem \ref{First theorem} we will show that for a general method of producing $\beta$-expansions the $\beta$-expansion generated for $x$ is eventually periodic if and only if $x\in\mathbb{Q}(\beta)$, generalising a theorem due to Schmidt \cite{Schmidt}. Theorem \ref{First theorem} will also provide the motivation for sufficient conditions under which the growth rate of $\beta$-expansions equals the Hausdorff dimension of the set of $\beta$-expansions, partially answering a question posed in \cite{Baker2}, moreover our method allows us to explicitly calculate these quantities. 

Before beginning our study of the sets $S_{\beta,m,n}(x)$ and $S_{\beta,m}(x)$ it is useful to recall the following, as we will see the subsequent theory will be important in understanding the topology of these sets.

To each $m\in\mathbb{N}$ and $\beta\in(1,m+1]$ we associate the set $$Y^{m}(\beta)=\Big\{\sum_{i=1}^{n} \epsilon_{i}\beta^{i}| \epsilon_{i}\in\{0,\ldots,m\}, n=0,1,\ldots \Big\}.$$ The elements of $Y^{m}(q)$ can be arranged into a strictly increasing sequence $y^{m}_{0}(\beta)<y^{m}_{1}(\beta)<y^{m}_{2}(\beta)<\ldots,$ tending to infinity. We define the quantities $$l^{m}(\beta)=\liminf_{k\to\infty} (y_{k+1}^{m}(\beta)-y_{k}^{m}(\beta)) \textrm{ and } L^{m}(\beta)=\limsup_{k\to\infty} (y_{k+1}^{m}(\beta)-y_{k}^{m}(\beta)).$$ 
These limits have been studied in great depth, to name but a few references we refer the reader to \cite{AkiKom,ErdosKomornik,Feng,SidSolom}. As we will see the quantities $l^{m}(\beta)$ and $L^{m}(\beta)$ will be intimately related to the topology of the sets $S_{\beta,m,n}(x)$ and $S_{\beta,m}(x).$

\section{Elementary Properties}
In this section we prove several elementary results relating the topology of $S_{\beta,m}(x)$ to the set of $\beta$-expansions. Following \cite{ErdosKomornik} we say that a $\beta$-expansion for $x$ is \textit{universal} if it contains all finite blocks of digits from $\{0,\ldots,m\}$. Similarly, we say that a point $x\in I_{\beta,m}$ is \textit{universal} if for any finite block of digits from $\{0,\ldots,m\}$ there exists a $\beta$-expansion for $x$ containing this block. The following propositions are immediate.

\begin{prop}
\label{Universal point}
$x\in I_{\beta,m}$ is universal if and only if $S_{\beta,m}(x)$ is dense in $I_{\beta,m}.$ 
\end{prop}

\begin{proof}
Assume $x$ is universal and let $\mathcal{I}$ be a nontrivial subinterval of $I_{\beta,m}$. Let $z\in int(\mathcal{I})$ and $(\delta_{i})_{i=1}^{\infty}$ be a $\beta$-expansion for $z,$ we consider the set $$\Gamma((\delta_{i})_{i=1}^{N})=\Big\{y\in I_{\beta,m}| (T_{\beta,\delta_{N}}\circ\ldots \circ T_{\beta,\delta_{1}})(y)\in I_{\beta,m}\Big\}.$$ This set is an interval of diameter $\frac{m}{\beta^{N}(\beta-1)}$ containing $z,$ since $z\in int(\mathcal{I})$ we have $\Gamma((\delta_{i})_{i=1}^{N})\subset \mathcal{I}$ for $N$ sufficiently large. As $x$ is universal there exists a $\beta$-expansion for $x$ containing the digits $(\delta_{1},\ldots, \delta_{N}),$ by an application of Lemma \ref{Bijection lemma} there must exist a finite sequence of maps $a$ such that $a(x)\in \Gamma((\delta_{i})_{i=1}^{N}).$ Therefore $S_{\beta,m}(x)\cap \mathcal{I}\neq \emptyset,$ as $\mathcal{I}$ was arbitrary we may conclude that $S_{\beta,m}(x)$ is dense in $I_{\beta,m}.$

We now prove the opposite implication. Suppose $S_{\beta,m}(x)$ is dense and let $\epsilon=(\epsilon_{1},\ldots,\epsilon_{n})$ be a finite block of digits from $\{0,\ldots,m\}.$ Let $$I_{\epsilon}=\Big\{y\in I_{\beta,m}| (T_{\beta,\epsilon_{n}}\circ\ldots \circ T_{\beta,\epsilon_{1}})(y)\in I_{\beta,m}\Big\}.$$ As $S_{\beta,m}(x)$ is dense there exists a finite sequence of maps $a$ such that $a(x)\in I_{\epsilon},$ therefore $(T_{\beta,\epsilon_{n}}\circ\ldots \circ T_{\beta,\epsilon_{1}}\circ a)(x)\in I_{\beta,m},$ our result follows from Lemma \ref{Bijection lemma}.

\end{proof}
\begin{prop}
\label{Universal expansion}
If $x\in I_{\beta,m}$ has a universal $\beta$-expansion then there exists $a\in\Omega_{\beta,m}(x)$ such that $\{(a_{n}\circ \ldots \circ a_{1})(x)|  n=1,2,\ldots\}$ is dense in $I_{\beta,m}.$
\end{prop}
The proof of this statement follows by a similar argument to the first part of Proposition \ref{Universal point} and is therefore omitted.

\begin{prop}
Every $x\in (0,\frac{m}{\beta-1})$ has a universal $\beta$-expansion if and only if $S_{\beta,m}(x)$ is dense in $I_{\beta,m}$ for all $x\in(0,\frac{m}{\beta-1}).$
\end{prop}
\begin{proof}
The rightwards implication is an immediate consequence of Proposition \ref{Universal expansion}. Let us suppose $S_{\beta,m}(x)$ is dense in $I_{\beta,m}$ for every $x\in (0,\frac{m}{\beta-1}).$ Let $\{B_{i}\}_{i=1}^{\infty}$ be an enumeration of all finite blocks of transformations from $\{T_{\beta,0},\ldots,T_{\beta,m}\}.$ For each $B_{i}$ there exists an interval $I_{B_{i}}$ such that, $B_{i}(y)\in I_{\beta,m}$ if and only if $y\in I_{B_{i}}.$ As $S_{\beta,m}(x)$ is dense in $I_{\beta,m}$ there exists a finite sequence of maps $a,$ such that $a(x)\in int(I_{B_{1}}),$ therefore $B_{1}\circ a(x) \in (0,\frac{m}{\beta-1}).$ Applying our hypothesis to $B_{1}\circ a(x)$ we can assert that there exists a finite sequence of maps $a_{1},$ such that $(B_{2}\circ a_{1} \circ B_{1} \circ a)(x)\in (0,\frac{m}{\beta-1}).$ Repeating this process arbitrarily many times we may construct an infinite sequence of maps containing all finite blocks from $\{T_{\beta,0},\ldots, T_{\beta,m}\}$, by Lemma \ref{Bijection lemma} our result follows.
\end{proof}

In \cite{ErdosKomornik} the following result was shown to hold.

\begin{thm}
\label{Big L theorem}
If $L^{m}(\beta)=0$ then every $x\in(0,\frac{m}{\beta-1})$ has a universal $\beta$-expansion.
\end{thm}

By Theorem \ref{Big L theorem} and the results presented in \cite{AkiKom}, \cite{Feng} and \cite{Sidorov2} the following theorem is immediate.

\begin{thm}
\label{First thm}
\begin{itemize}
	\item Let $\beta\in(1,2^{1/3}]$ and assume $\beta$ is not a Pisot number, then $S_{\beta,1}(x)$ is dense in $I_{\beta,1}$ for every $x\in(0,\frac{1}{\beta-1})$.
	\item Let $\beta\in(2^{1/3},2^{1/2})$ and assume that $\beta^{2}$ is not a Pisot number, then $S_{\beta,1}(x)$ is dense in $I_{\beta,1}$ for every $x\in(0,\frac{1}{\beta-1}).$
	\item Let $\beta\in(1,2),$ then for almost every $x\in I_{\beta,1}$ there exists $a\in\Omega_{\beta,1}(x)$ such that $\{a_{n}\circ \ldots \circ a_{1}(x)|  n=1,2,\ldots\}$ is dense in $I_{\beta,1}.$
\end{itemize}
\end{thm}

\section{The Pisot case}

In this section we study the case when $\beta$ is Pisot. As well as proving Theorem \ref{First theorem} we will show that the following result holds.

\begin{prop}
\label{Pisot bounded}
Let $\beta$ be a Pisot number and $x\in I_{\beta,m}$, then $Card(S_{\beta,m,n}(x))$ can be bounded above by some constant depending only on $\beta$ and $m$.
\end{prop}

To prove Theorem \ref{First theorem} and Proposition \ref{Pisot bounded} we require the following theorem due to Garsia \cite{Garsia}.
\begin{thm}
\label{Garcia thm}
Let $m\in\mathbb{N}$ and $\beta\in(1,m+1]$ be a Pisot number, then $Y^{m}(\beta)$ is uniformly discrete, i.e., there exists $\Delta(\beta,m)>0$ such that $|x-y|\geq \Delta(\beta,m)$ for all $x,y\in Y^{m}(\beta)$ such that $x\neq y.$
\end{thm}
\begin{proof}[Proof of Theorem \ref{First theorem}]
Suppose $S_{\beta,m}(x)$ is finite, in this case there exists $n<n',$ $a\in \Omega_{\beta,m,n}(x)$ and $a'\in \Omega_{\beta,m,n'}(x)$ such that $a(x)=a'(x)$. It follows that there exists $(\epsilon_{i})_{i=0}^{n-1}\in \{0,\ldots,m\}^{n}$ and $(\epsilon_{i}')_{i=0}^{n'-1}\in \{0,\ldots,m\}^{n'}$ satisfying $$\beta^{n}x-\sum_{i=0}^{n-1}\epsilon_{i}\beta^{i}=\beta^{n'}x-\sum_{i=0}^{n'-1}\epsilon_{i}'\beta^{i},$$ from which it is a simple consequence that $x\in\mathbb{Q}(\beta).$ It remains to show that the opposite implication holds. Let $x\in \mathbb{Q}(\beta)$, then $x=p(\beta)/n_{1}$ for some $n_{1}\in\mathbb{N}$ and $p(\beta)=\sum_{i=0}^{d}\delta_{i}\beta^{i}$ where $\delta_{i}\in\mathbb{Z}.$ Let $\epsilon>0$, we assume for a contradiction that $S_{\beta,m}(x)$ is infinite. If $S_{\beta,m}(x)$ is infinite then there exists $j,j'\in\mathbb{N}$ such that 
\begin{equation}
\label{equation 1}
0<\Big|\Big(\frac{\beta^{j}p(\beta)}{n_{1}}-\sum_{i=0}^{j-1}\epsilon_{i}\beta^{i}\Big)-\Big(\frac{\beta^{j'}p(\beta)}{n_{1}}-\sum_{i=0}^{j'-1}\epsilon'_{i}\beta^{i}\Big)\Big|<\frac{\epsilon}{n_{1}},
\end{equation}for some $(\epsilon_{i})_{i=0}^{j-1}\in\{0,\ldots,m\}^{j}$ and $(\epsilon'_{i})_{i=0}^{j'-1}\in\{0,\ldots,m\}^{j'}.$ By an abuse of notation we can rewrite (\ref{equation 1}) as
\begin{equation}
\label{equation 2}
\Big|\frac{\beta^{j}p(\beta)-\beta^{j'}p(\beta)}{n_{1}}-\sum_{i=0}^{k}\epsilon_{i}\beta^{i}\Big|<\frac{\epsilon}{n_{1}},
\end{equation} where $k=\max\{j-1,j'-1\}$ and $(\epsilon_{i})_{i=0}^{k}\in\{-m,\ldots,m\}^{k+1}.$ Multiplying through by $n_{1}$ we can rewrite (\ref{equation 2}) as
\begin{equation}
\label{equation 3}
\Big|\beta^{j}p(\beta)-\beta^{j'}p(\beta)-n_{1}\sum_{i=0}^{k}\epsilon_{i}\beta^{i}\Big|<{\epsilon}.
\end{equation}Let $n_{2}=\max\{|\delta_{i}|\}$ and $L=\max\{d+j,d+j'\},$ collecting positive and negative terms we can rewrite (\ref{equation 3}) as 
\begin{equation}
\label{equation 4}
\Big|\sum_{i=0}^{L}\omega_{i}\beta^{i}-\sum_{i=0}^{L}\omega_{i}'\beta^{i}\Big|<{\epsilon},
\end{equation} for some $(\omega_{i})_{i=0}^{L},(\omega_{i'})_{i=0}^{L}\in \{0,\ldots, n_{1}m+n_{2}\}^{L+1}.$ We remark that $n_{1}m+n_{2}$ has no dependence on $j$ and $j'$ and if we take $\epsilon=\Delta(\beta,n_{1}m+n_{2}),$ then by Theorem \ref{Garcia thm} we have a contradiction.
\end{proof}
\begin{proof}[Proof of Proposition \ref{Pisot bounded}]
Let $m\in\mathbb{N}$ and $\beta\in(1,m+1]$ be a Pisot number, for each $x\in I_{\beta,m}$ we can rewrite $S_{\beta,m,n}(x)$ as 
$$S_{\beta,m,n}(x)=\Big\{y\in I_{\beta,m}|y=\beta^{n}x-\sum_{i=0}^{n-1}\epsilon_{i}\beta^{i} \textrm{ where } \epsilon_{i}\in\{0,\ldots,m\}\Big\}.$$ Let $z,z'\in S_{\beta,m,n}(x)$ and $z\neq z',$ then $|z-z'|= |\sum_{i=0}^{n-1}\epsilon_{i}\beta^{i}-\sum_{i=0}^{n-1}\epsilon_{i}'\beta^{i}|$ for some $(\epsilon_{i})_{i=0}^{n-1}, (\epsilon_{i}')_{i=0}^{n-1}\in\{0,\ldots,m\}^{n}.$ By Theorem \ref{Garcia thm} $|z-z'|\geq \Delta(\beta,m)$ and therefore $$Card(S_{\beta,m,n}(x))\leq \Big[\frac{\frac{m}{\beta-1}}{\Delta(\beta,m)}\Big]+1.$$
\end{proof}
\begin{remark}
In \cite{Sidorov2} it was shown that for $\beta\in(1,2)$ almost every $x\in I_{\beta,1}$ has a universal $\beta$-expansion. By Proposition $\ref{Universal expansion}$ it follows that $S_{\beta,1}(x)$ is dense in $I_{\beta,m}$ for almost every $x\in I_{\beta,1}.$ We might expect $Card(S_{\beta,1,n}(x))\to \infty$ as $n\to\infty$ for almost every $x$. However, by Proposition \ref{Pisot bounded} $Card(S_{\beta,1,n}(x))$ can be bounded above when $\beta$ is Pisot for all $x\in I_{\beta,m}$.
\end{remark}
 The following corollary is an immediate consequence of Theorem \ref{First theorem} and Propositions \ref{Universal point} and \ref{Universal expansion}.
\begin{cor}
Let $\beta$ be Pisot, if $x\in\mathbb{Q}(\beta)$ then $x$ cannot be universal or have a universal $\beta$-expansion.
\end{cor} This generalises a result in \cite{ErdosKomornik} where it was shown that if $\beta$ is Pisot then $1$ cannot have a universal $\beta$-expansion.

\section{Generalisation of Schmidt's theorem}
In this section we generalise a theorem due to Schmidt \cite{Schmidt}. Before stating our theorem and Schmidt's it is necessary to establish the following. Let $\mathcal{A}=\{B_{k}\}_{k=1}^{2^{m+1}-1}=\mathcal{P}(\{0,\ldots,m\})\setminus\ \{\emptyset\},$ where $\mathcal{P}(\{0,\ldots,m\})$ denotes the powerset of $\{0,\ldots,m\}.$ We set $$B(x):=\Big\{i\in\{0,\ldots,m\} : T_{\beta,i}(x)\in I_{\beta,m}\Big\},$$ and to each $B_{k}\in \mathcal{A}$ we associate the set $$I_{k,\beta,m}=\Big\{x\in I_{\beta,m}: B(x)=B_{k}\Big\}.$$

We remark that for many $B_{k}$ the corresponding set $I_{k,\beta,m}$ will be empty, however for our purposes this will not be important. We also remark that the set $\{I_{k,\beta,m}\}_{k=1}^{2^{m+1}-1}$ forms a partition of $I_{\beta,m}.$ When $I_{k,\beta,m}$ is non-empty it will be an interval; possibly open, closed or neither. 
\begin{example}
Let $m=1$ and $\beta\in(1,2],$ in this case $\mathcal{A}=\{B_{k}\}_{k=1}^{3}=\{ \{0\},\{1\},\{0,1\}\}$ and $I_{1,\beta,m}=[0,\frac{1}{\beta}),$ $I_{2,\beta,m}=(\frac{1}{\beta(\beta-1)},\frac{1}{\beta-1}]$ and $I_{3,\beta,m}=[\frac{1}{\beta},\frac{1}{\beta(\beta-1)}].$
\end{example}

Suppose $B_{k}=\{\delta_{k,l}\}_{l=1}^{p(k)},$ we let $\{A_{k,l}\}_{l=1}^{p(k)}$ be a partition of $I_{k,\beta,m}.$ As  $\{I_{k,\beta,m}\}_{k=1}^{2^{m+1}-1}$ is a partition of $I_{\beta,m}$ then $\{A_{k,l}\}_{k,l}$ is also a partition of $I_{\beta,m}.$ We define the \textit{expansion generating function associated to $\{A_{k,l}\}_{k,l}$} to be the function $F:I_{\beta,m}\to I_{\beta,m},$ where $F(x)=T_{\beta,\delta_{k,l}}(x)$ if  $x\in A_{k,l}.$ As $\{A_{k,l}\}_{k,l}$ is a partition of $I_{\beta,m}$ the function $F$ is well defined. We refer to $F$ as the expansion generating function associated to $\{A_{k,l}\}_{k,l}$ because by repeatedly iterating our map $F$ it associates to each $x\in I_{\beta,m}$ a unique element of $\Omega_{\beta,m}(x),$ by Lemma \ref{Bijection lemma} this corresponds to a unique $\beta$-expansion for $x$. We define this unique $\beta$-expansion to be the \textit{$\beta$-expansion generated by $F$}. Intuitively we think of the partition $\{A_{k,l}\}_{k,l}$ as a collection of rules under which whenever we have a choice of maps from $\{T_{\beta,0},\ldots, T_{\beta,m}\}$ satisfying $T_{\beta,i}(x)\in I_{\beta,m},$ our rules decide which of these maps we perform. We remark that the elements of $\{A_{k,l}\}_{k,l}$ may have an exotic structure, they need not be intervals or even measurable sets. We refer the reader to Figure \ref{fig1} for a diagram illustrating a typical expansion generating function in the case where $m=1$ and $\beta\in(1,2).$
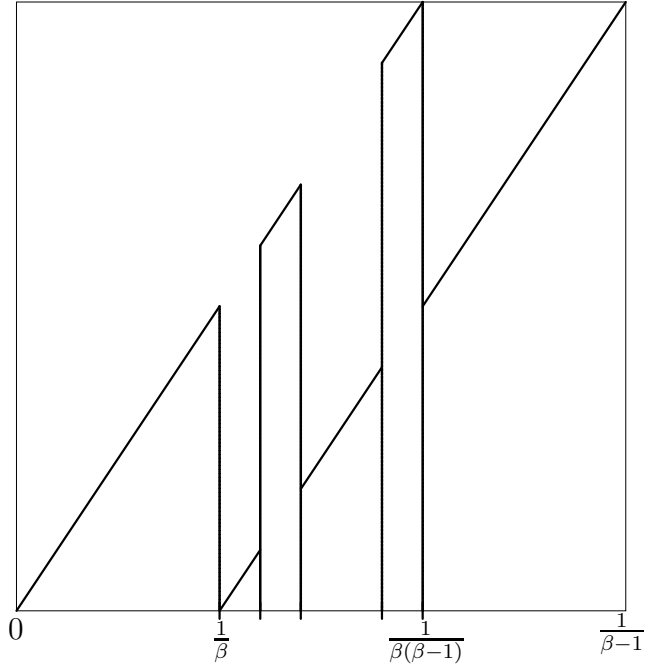
\begin{figure}[t]
\centering \unitlength=0.54mm
\begin{picture}(150,160)(0,-10)
\thinlines
\path(115,0)(0,0)(0,150)(150,150)(150,0)(115,0)
\put(-2,-7){$0$}
\put(143,-7){$\frac{1}{\beta-1}$}
\put(48,-9){$\frac{1}{\beta}$}
\put(91,-9){$\frac{1}{\beta(\beta-1)}$}

\thicklines
\path(0,0)(50,75)
\path(50,0)(60,15)
\path(60,90)(70,105)
\path(70,30)(90,60)
\path(90,135)(100,150)
\path(50,2)(50,-2)
\path(100,2)(100,-2)
\path(60,2)(60,-2)
\path(70,2)(70,-2)
\path(90,2)(90,-2)
\path(100,75)(150,150)
\dottedline(60,90)(60,0)
\dottedline(70,105)(70,0)
\dottedline(90,135)(90,0)
\dottedline(100,150)(100,0)
\dottedline(50,75)(50,0)
\end{picture}
\caption{A typical expansion generating function for $m=1$ and $\beta\in(1,2)$}
    \label{fig1}
\end{figure}

\begin{remark}
\label{Greedy lazy}
If $A_{k,\max\{\delta_{k,l}\}}=I_{k,\beta,m}$ for each $1\leq k\leq 2^{m+1}-1,$ then 

\[ F(x) = \left\{ \begin{array}{ll}
         \beta x(\textrm{mod }1)& \mbox{if $x\in[0,\frac{m}{\beta}] $}\\
        \beta x-m & \mbox{if $x\in(\frac{m}{\beta},\frac{m}{\beta-1}]$.}\end{array} \right. \] The $\beta$-expansion generated by this function is the greedy expansion. If $A_{k,\min\{\delta_{k,l}\}}=I_{k,\beta,m}$ then the $\beta$-expansion generated by $F$ is the lazy expansion. We refer the reader to \cite{Sidorov3} for the relevant details regarding greedy and lazy expansions.
\end{remark}
When $A_{k,\max\{\delta_{k,l}\}}=I_{k,\beta,m}$ as in Remark \ref{Greedy lazy} we denote the \textit{expansion generating function associated to $\{A_{k,l}\}_{k,l}$} by $F_{greedy},$ moreover we let $Pre(F_{greedy})$ denote the set of pre-periodic points of $F_{greedy}.$ In \cite{Schmidt} the following theorem was shown to hold.

\begin{thm}
\label{Schmidt theorem}
Let $\beta$ be a Pisot number, consider $F_{greedy}:[0,1)\to[0,1),$ then $Pre(F_{greedy})=\mathbb{Q}(\beta)\cap [0,1).$
\end{thm}
 It is not difficult to show that this result can be extended to the case when $F_{greedy}:I_{\beta,m}\to I_{\beta,m}.$ We generalise this result as follows.

\begin{thm}
\label{Schmidt generalisation}
Let $\beta$ be a Pisot number, then $Pre(F)=\mathbb{Q}(\beta)\cap I_{\beta,m}$ for any expansion generating function $F$.
\end{thm}
\begin{proof}
Let $x\in Pre(F),$ by Lemma \ref{Bijection lemma} there exists an eventually periodic $\beta$-expansion for $x,$ manipulating this expansion using standard techniques for geometric series we can conclude that $Pre(F)\subset\mathbb{Q}(\beta)\cap I_{\beta,m}$. We now show the opposite inclusion, let $x\in \mathbb{Q}(\beta)\cap I_{\beta,m}$ and $F$ be an expansion generating function corresponding to some $\{A_{k,l}\}_{k,l}$. Each succesive iterate of the map $F$ is an element of $S_{\beta,m}(x).$ By Theorem \ref{First theorem} we have that $S_{\beta,m}(x)$ is finite and therefore there exists $N,N'\in\mathbb{N},$ such that $F^{N}(x)=F^{N'}(x),$ therefore $x\in Pre(F).$ 
\end{proof}
In \cite{KalSte} a version of Theorem \ref{Schmidt generalisation} was shown to hold for a more general class of $\beta$-expansion, however, this was under the weaker assumption that the elements of $\{A_{k,l}\}_{k,l}$ were intervals.

The following corollary is an immediate consequence of Theorem \ref{Schmidt generalisation} and Lemma \ref{Bijection lemma}.

\begin{cor}
Let $\beta$ be a Pisot number, then for any expansion generating function $F$ the $\beta$-expansion generated by $F$ is eventually periodic if and only if $x\in \mathbb{Q}(\beta)\cap I_{\beta,m}.$ 
\end{cor}

\section{The growth rate and Hausdorff dimension of $\Sigma_{\beta,m}(x)$}
In this section we study the growth rate and Hausdorff dimension of the set of $\beta$-expansions. Let $\beta\in (1,m+1]$ be some arbitrary number not necessarily Pisot, we assume that $x\in I_{\beta,m}$ satisfies $S_{\beta,m}(x)=\{\gamma_{j}\}_{j=1}^{k}$. Our motivation for this finiteness condition comes from Theorem \ref{First theorem} but as we will see the following results do not require any assumptions on the algebraic properties of $\beta$ or $x.$ However, we remark that we are unaware of any non-trivial examples where $\beta$ is not Pisot and there exists $x\in I_{\beta,m}$ such that $S_{\beta,m}(x)$ is finite. In what follows we assume that our enumeration of the set $\{\gamma_{j}\}_{j=1}^{k}$ is such that $\gamma_{1}=x.$

Recall the following from \cite{Baker2}, let $$\mathcal{N}_{\beta,m,n}(x)=Card(\mathcal{E}_{\beta,m,n}(x))$$ and define the \textit{growth rate of $\beta$-expansions} to be $$\lim_{n\to\infty} \frac{\log_{m+1}\mathcal{N}_{\beta,m,n}(x)}{n},$$ when this limit exists. When this limit does not exist we can consider the \textit{lower and upper growth rates of $\beta$-expansions}, these are defined to be $$\liminf_{n\to\infty} \frac{\log_{m+1}\mathcal{N}_{\beta,m,n}(x)}{n}\textrm{ and }\limsup_{n\to\infty} \frac{\log_{m+1}\mathcal{N}_{\beta,m,n}(x)}{n}$$ respectively. 
We endow $\{0,\ldots,m\}^{\mathbb{N}}$ with the metric $d(\cdot,\cdot)$ defined as follows:
\[ d(x,y) = \left\{ \begin{array}{ll}
         (m+1)^{-n(x,y)} & \mbox{if $x\neq y,$ where $n(x,y)=\inf \{i:x_{i}\neq y_{i}\} $}\\
        0 & \mbox{if $x=y$.}\end{array} \right. \]We can consider the Hausdorff dimension of $\Sigma_{\beta,m}(x)$ with respect to this metric. It is a simple exercise to show that following inequalities hold:
\begin{equation}
\label{dimension inequality}
\dim_{H}(\Sigma_{\beta,m}(x))\leq \liminf_{n\to\infty} \frac{\log_{m+1}\mathcal{N}_{\beta,m,n}(x)}{n}\leq\limsup_{n\to\infty} \frac{\log_{m+1}\mathcal{N}_{\beta,m,n}(x)}{n}.
\end{equation} 
These quantities were studied in \cite{Baker,Baker2,FengSid,Kempton}. In \cite{Baker} it was shown that for each $m\in\mathbb{N}$ there exists $\mathcal{G}(m)\in\mathbb{R}$ such that, for $\beta\in(1,\mathcal{G}(m))$ and $x\in(0,\frac{m}{\beta-1})$ the Hausdorff dimension of $\Sigma_{\beta,m}(x)$ can be bounded below by some strictly positive function depending only on $\beta$ and $m.$

We define the \textit{transition matrix associated to $x$} to be the $k\times k$ matrix $A,$ where $A$ satisfies
\[ (A)_{q,j} = \left\{ \begin{array}{ll}
         1 & \mbox{if there exists $i\in\{0,\ldots,m\}$ such that $T_{\beta,i}(\gamma_{q})=\gamma_{j}$}\\
        0 & \mbox{otherwise.}\end{array} \right. \]

Let $N(\gamma_{q},\gamma_{j},n)=Card(\{a\in \Omega_{\beta,m,n}(x) \textrm{ such that } a(\gamma_{q})=\gamma_{j}\}).$ The following proposition is immediate.

\begin{prop}
\label{Counting prop}
Let $e_{j}$ denote the $(k\times 1)$ column vector that is $1$ in the $j$-th entry and $0$ in all other entries. Then $N(\gamma_{q},\gamma_{j},n)=(A^{n}e_{j})_{q}$ and $\mathcal{N}_{\beta,m,n}(\gamma_{q})=\sum_{j=1}^{k}(A^{n}e_{j})_{q}.$

\end{prop}

\begin{proof}
As $\mathcal{N}_{\beta,m,n}(\gamma_{q})=\sum_{j=1}^{k}N(\gamma_{q},\gamma_{j},n)$ it suffices to show that the first statement holds. It is a standard inductive argument to show that $$N(\gamma_{q},\gamma_{j},n)=(A^{n})_{q,j},$$ our result then follows from the observation that $(A^{n})_{q,j}=(A^{n}e_{j})_{q}$

\end{proof}

We now give conditions under which we have equality in (\ref{dimension inequality}) and can explicitly compute $\dim_{H}(\Sigma_{\beta,m}(x))$ and the growth rate of $\beta$-expansions. Let $A$ be the transition matrix associated to $\{\gamma_{j}\}_{j=1}^{k}$ as above. As $A$ is a non-negative matrix with non-zero entries it has a positive real eigenvalue $\alpha$ with non-negative eigenvector $v_{\alpha},$ such that $Spec(A)=\alpha.$ It maybe the case that there exists other possibly complex eigenvalues $\alpha_{i}$ such that $|\alpha_{i}|=\alpha.$ This is the case we want to avoid, as such we introduce the following condition. We say that $A$ satisfies \textit{condition $1$} if $A$ has a positive real eigenvalue $\alpha$  with non-negative eigenvector $v_{\alpha}$ such that $|\alpha_{i}|<\alpha$ for all other eigenvalues. Condition $1$ is satisfied if for every $\gamma_{i},\gamma_{j}\in S_{\beta,m}(x)$ there exists a finite sequence of maps $a_{i,j}$ such that $a_{i,j}(\gamma_{i})=\gamma_{j},$ and if $p_{i}$ is the minimum number of transformations required to map $\gamma_{i}$ to $\gamma_{i},$ then $gcd(\{p_{i}\})=1$. This is a consequence of the Perron-Frobenius theorem for primitive matrices.

\begin{thm}
\label{Dimension thm}
Let $m\in\mathbb{N},$ $\beta\in(1,m+1]$ and $x\in I_{\beta,m}.$ Assume $S_{\beta,m}(x)=\{\gamma_{j}\}_{j=1}^{k}$ and the transition matrix $A$ associated to $x$ satisfies condition $1,$ then
$$\dim_{H}(\Sigma_{\beta,m}(x))=\lim_{n\to\infty} \frac{\log_{m+1}\mathcal{N}_{\beta,m,n}(x)}{n}=\log_{m+1}\alpha.$$ 
\end{thm}
Before proving Theorem \ref{Dimension thm} we require the following lemma.
\begin{lemma}
\label{Growth lemma}
Under the hypothesis of Theorem \ref{Dimension thm} there exists $C_{x}>0$ and $D>0$ such that $$C_{x}\alpha^{n}\leq \mathcal{N}_{\beta,m,n}(x)$$ and $$\mathcal{N}_{\beta,m,n}(\gamma_{j})\leq D\alpha^{n}$$ for all $\gamma_{j}\in S_{\beta,m}(x)$ and $n\in\mathbb{N}$.
\end{lemma}
\begin{proof}
The existence of $D$ follows by a simple linear algebra argument and Proposition \ref{Counting prop}. It remains to show the existence of $C_{x}.$ Let $i\in\{1,\ldots,k\}$ be such that $(v_{\alpha})_{i}>0$. By Proposition \ref{Counting prop} we have that 
\begin{align*}
\mathcal{N}_{\beta,m,n}(\gamma_{i})\geq (A^{n}e_{i})_{i}&=(A^{n}(proj_{v_{\alpha}}(e_{i}))+A^{n}(e_{i}-proj_{v_{\alpha}}(e_{i})))_{i}\\
&=(\alpha^{n}(proj_{v_{\alpha}}(e_{i}))+A^{n}(e_{i}-proj_{v_{\alpha}}(e_{i})))_{i}\\
&=(\alpha^{n}(proj_{v_{\alpha}}(e_{i})))_{i}+(A^{n}(e_{i}-proj_{v_{\alpha}}(e_{i})))_{i}.
\end{align*} Here $proj_{v_{\alpha}}$ denotes the projection onto the eigenvector $v_{\alpha}$, since $(v_{\alpha})_{i}>0$ it follows that $proj_{v_{\alpha}}(e_{i})$ is nonzero and by condition 1 there exists $C>0$ such that 
\begin{equation}
\label{Some equation}
\mathcal{N}_{\beta,m,n}(\gamma_{i})\geq C \alpha^{n}.
\end{equation} There exists a sequence of transformations $a$ of length $n_{i}$ such that $a(x)=\gamma_{i},$ therefore $\mathcal{N}_{\beta,m,n+n_{i}}(x)\geq\mathcal{N}_{\beta,m,n}(\gamma_{i}).$ By (\ref{Some equation}) we can conclude that $\mathcal{N}_{\beta,m,n+n_{i}}(x)\geq C \alpha^{n},$ for all $n\in\mathbb{N},$ our result follows.
\end{proof}
Applying Lemma \ref{Growth lemma} we can conclude that $\frac{\log_{m+1}\mathcal{N}_{\beta,m,n}(x)}{n}=\alpha.$ By (\ref{dimension inequality}) to prove Theorem \ref{Dimension thm} it suffices to show that $\dim_{H}(\Sigma_{\beta,m}(x))\geq \alpha.$ Our method of proof is analogous to that given in \cite{Baker}, which is based upon Example $2.7$ of \cite{Falconer}.
\begin{proof}[Proof of Theorem \ref{Dimension thm}]
As $\Sigma_{\beta,m}(x)$ is a compact set we may restrict to finite covers. Let $\{U_{n}\}_{n=1}^{N}$ be a finite cover of $\Sigma_{\beta,m}(x),$ without loss of generality we may assume that $\textrm{Diam}(U_{n})< \frac{1}{m+1},$ as such for each $U_{n}$ there exists $l(n)\in\mathbb{N}$ such that $$(m+1)^{-(l(n)+1)}\leq \textrm{Diam}(U_{n})< (m+1)^{-l(n)}.$$ It follows that there exists $z^{(n)}\in \{0,\ldots,m\}^{l(n)}$ such that, $y_{i}=z^{(n)}_{i}$ for $1\leq i \leq l(n),$ for all $y\in U_{n}.$ We may assume that $z^{(n)}\in \mathcal{E}_{\beta,m,l(n)}(x),$ if we supposed otherwise then $\Sigma_{\beta,m}(x)\cap U_{n}=\emptyset$ and we can remove $U_{n}$ from our cover. We denote by $C_{n}$ the set of sequences in $\{0,\ldots,m\}^{\mathbb{N}}$ whose first $l(n)$ entries agree with $z^{(n)},$ i.e. $$C_{n}=\Big\{(\epsilon_{i})_{i=1}^{\infty}\in \{0,\ldots,m\}^{\mathbb{N}}: \epsilon_{i}= z^{(n)}_{i}\textrm{ for } 1\leq i\leq l(n)\Big\}.$$ Clearly $U_{n}\subset C_{n}$ and therefore the set $\{C_{n}\}_{n=1}^{N}$ is a cover of $\Sigma_{\beta,m}(x).$

Since there are only finitely many elements in our cover there exists $J\in\mathbb{N}$ such that $(m+1)^{-J}\leq \textrm{Diam}(U_{n})$ for all $n$. We consider the set $\mathcal{E}_{\beta,m,J}(x).$ Since $\{C_{n}\}_{n=1}^{N}$ is a cover of $\Sigma_{\beta,m}(x)$ each $a\in\mathcal{E}_{\beta,m,J}(x)$ satisfies $a_{i}=z^{(n)}_{i}$ for $1\leq i \leq l(n),$ for some $n$. Therefore $$\mathcal{N}_{\beta,m,J}(x)\leq \sum_{n=1}^{N} Card(\{a\in \mathcal{E}_{\beta,m,J}(x): a_{i}=z^{(n)}_{i} \textrm{ for } 1\leq i\leq l(n)\}).$$Applying Lemma \ref{Growth lemma} the following inequality is immediate;
\begin{equation}
\label{First bounds}
C_{x}\alpha^{J}\leq \sum_{n=1}^{N} Card(\{a\in \mathcal{E}_{\beta,m,J}(x): a_{i}=z^{(n)}_{i} \textrm{ for } 1\leq i\leq l(n)\}).
\end{equation} Each element of the set $\{a\in \mathcal{E}_{\beta,m,J}(x): a_{i}=z^{(n)}_{i} \textrm{ for } 1\leq i\leq l(n)\}$ can be identified with a prefix of length $J-l(n)$ for some element of $\{\gamma_{j}\}_{j=1}^{k},$ this is a simple consequence of Lemma \ref{Bijection lemma}. We may therefore apply the second bound from Lemma \ref{Growth lemma}, 
\begin{align*}
\sum_{n=1}^{N} Card(\{a\in \mathcal{E}_{\beta,m,J}(x): a_{i}=z^{(n)}_{i} \textrm{ for } 1\leq i\leq l(n)\})&\leq \sum_{n=1}^{N} D\alpha^{J-l(n)}\\
&= D\alpha^{J+1}\sum_{n=1}^{N} (m+1)^{-(l(n)+1)\log_{m+1}\alpha}\\
&\leq D\alpha^{J+1}\sum_{n=1}^{N} \textrm{Diam}(U_{n})^{\log_{m+1}\alpha}.
\end{align*} Combining the above with (\ref{First bounds}) we have that the following inequality holds; $$C_{x}\alpha^{J}\leq D\alpha^{J+1}\sum_{n=1}^{N} \textrm{Diam}(U_{n})^{\log_{m+1}\alpha}.$$ Dividing through by $D\alpha^{J+1}$ yields $$\sum_{n=1}^{N} \textrm{Diam}(U_{n})^{\log_{m+1}\alpha} \geq \frac{C_{x}}{D\alpha},$$ the right hand side is a constant greater than zero that does not depend on our choice of cover. It follows that $\dim_{H}(\Sigma_{\beta,m}(x))\geq\log_{m+1}\alpha.$ 
\end{proof}

\section{Explicit calculation}
In this section we show how we can explicitly compute $\dim_{H}(\Sigma_{\beta,m}(x))$ and the growth rate of $\beta$ expansions for some $x,\beta$ and $m$. In what follows we assume $\beta\approx 1.53416$ is the Pisot number whose minimial polynomial is given by $z^5-z^3-z^2-z-1,$ $x=\frac{1}{\beta^{2}-1}$ and $m=1.$ It is a simple computation to show that
\begin{align*}
S_{\beta,1}\Big(\frac{1}{\beta^{2}-1}\Big)=\{\gamma_{j}\}_{j=1}^{10}=\Big\{&\frac{1}{\beta^{2}-1},\frac{\beta}{\beta^{2}-1},\frac{1+\beta-\beta^{2}}{\beta^{2}-1},\frac{\beta+\beta^{2}-\beta^{3}}{\beta^{2}-1},\frac{\beta^{2}+\beta^{3}-\beta^{4}}{\beta^{2}-1}\\
&\frac{\beta^{3}+\beta^{4}-\beta^{5}}{\beta^{2}-1},\frac{\beta^{2}}{\beta^{2}-1},\frac{\beta^{3}-\beta^{2}+1}{\beta^{2}-1},\frac{\beta^{4}-\beta^{3}-\beta^{2}+\beta+1}{\beta^{2}-1},\\
&\frac{\beta^{5}-\beta^{4}-\beta^{3}+\beta+1}{\beta^{2}-1}\Big\}
\end{align*} and the matrix $A$ is the $10\times 10$ matrix of the form

\[ A=\left( \begin{array}{cccccccccc}
0&1&1&0&0&0&0&0&0&0\\
1&0&0&0&0&0&1&0&0&0\\
0&0&0&1&0&0&0&0&0&0\\
0&0&0&0&1&0&0&0&0&0\\
0&0&0&0&0&1&0&0&0&0\\
1&0&0&0&0&0&0&0&0&0\\
0&0&0&0&0&0&0&1&0&0\\
0&0&0&0&0&0&0&0&1&0\\
0&0&0&0&0&0&0&0&0&1\\
0&1&0&0&0&0&0&0&0&0
\end{array} \right)\] This matrix has maximal eigenvalue $\kappa\approx 1.325$ with strictly positive eigenvector 
$$v_{\kappa}\approx(0.478,0.478,0.155,0.206,0.273,0.361,0.155,0.206,0.273,0.361).$$ By Theorem \ref{Dimension thm} it follows that 
$$\dim_{H}(\Sigma_{\beta,1}(x))=\lim_{n\to\infty} \frac{\log_{2}\mathcal{N}_{\beta,1,n}(x)}{n}\approx \log_{2} 1.325\approx 0.40599\ldots.$$ This result in fact holds for all $\gamma_{j}\in S_{\beta,1}(\frac{1}{\beta^{2}-1}).$ 
\begin{remark}
In \cite{FengSid} the authors show that if $\beta\in(1,2)$ is a Pisot number, almost every $x\in I_{\beta,1}$ satisfies $$\lim_{n\to\infty} \frac{\log_{2}\mathcal{N}_{\beta,1,n}(x)}{n}=\gamma,$$ where $\gamma<\log_{2}(\frac{2}{\beta}).$ However, when $\beta$ is as above and $x=\frac{1}{\beta^{2}-1}$ we have that $$\lim_{n\to\infty} \frac{\log_{2}\mathcal{N}_{\beta,1,n}(\frac{1}{\beta^{2}-1})}{n}=\log_{2} 1.325>\log_{2}\Big(\frac{2}{\beta}\Big).$$ Their bound cannot therefore be extended to all $x\in (0,\frac{1}{\beta-1}).$
\end{remark}

\noindent \textbf{Acknowledgements} The author would like to thank Nikita Sidorov for his feedback and Kevin Hare for suggesting a shorter proof of Theorem \ref{First theorem}. We would also like to thank the anonymous referee for their useful comments.


\begin{thebibliography}{100}
\bibitem{AkiKom} S. Akiyama and V. Komornik, Discrete Spectra and Pisot numbers, 	arXiv:1103.4508 [math.NT].
\bibitem{Baker} S. Baker, Generalised golden ratios over integer alphabets, arXiv:1210.8397 [math.DS].
\bibitem{Baker2} S. Baker, The growth rate and dimension theory of beta-expansions, Fund. Math. 219 (2012), 271--285.
\bibitem{Erdos} P. Erd\H{o}s, I. Jo\'{o} and V. Komornik, Characterization of the unique expansions $1 =\sum_{i=1}^{\infty}q^{-n_{i}}$ and
related problems, Bull. Soc. Math. Fr. 118 (1990), 377--390.
\bibitem{ErdosKomornik} P. Erd\H{o}s and V. Komornik, Developments in non-integer bases, Acta Math. Hungar. 79 (1998), no. 1-2, 57–-83.
\bibitem{Falconer} K. Falconer. Fractal Geometry: Mathematical Foundation and Applications. John Wiley, Chichester, 1990.
\bibitem{Feng} D. J. Feng, On the topology of polynomials with bounded integer coefficients, 	arXiv:1109.1407 [math.NT].
\bibitem{FengSid} D. J. Feng and N. Sidorov, Growth rate for beta-expansions, Monatsh. Math. 162 (2011), no. 1, 41--60. 
\bibitem{Garsia} A. Garsia, Arithmetic properties of Bernoulli convolutions, Trans. Amer. Math. Soc. 102 1962, 409--432.
\bibitem{GlenSid} P. Glendinning and N. Sidorov, Unique representations of real numbers in non-integer bases, Math. Res. Letters 8 (2001), 535--543.
\bibitem{KalSte} C. Kalle and W. Steiner, Beta-expansions, natural extensions and multiple tilings associated with Pisot units, Trans. Amer. Math. Soc. 364 (2012), no. 5, 2281--2318. 
\bibitem{Kempton} T. Kempton, Counting $\beta$-expansions and the absolute continuity of Bernoulli convolutions, Preprint.
\bibitem{KoLiDe} D. Kong, W. Li and F. Dekking, Intersections of homogeneous Cantor sets and beta-expansions, Nonlinearity 23 (2010), no. 11, 2815--2834. 
\bibitem{Schmidt} K. Schmidt, On periodic expansions of Pisot numbers and Salem numbers, Bull. London Math. Soc. 12 (1980), no. 4, 269–-278. 
\bibitem{Sidorov1} N. Sidorov, Almost every number has a continuum of $\beta$-expansions. Amer. Math. Monthly 110 (2003), no. 9, 838--842.
\bibitem{Sidorov3} N. Sidorov, Arithmetic dynamics, Topics in dynamics and ergodic theory, 145–-189, 
London Math. Soc. Lecture Note Ser., 310, Cambridge Univ. Press, Cambridge, 2003. 
\bibitem{Sidorov4} N. Sidorov, Combinatorics of linear iterated function systems with overlaps, Nonlinearity 20 (2007), no. 5, 1299--1312.
\bibitem{Sidorov2} N. Sidorov, Universal $\beta$-expansions, Period. Math. Hungar. 47 (2003), no. 1-2, 221--231.
\bibitem{SidSolom} N. Sidorov and B. Solomyak, On the topology of sums in powers of an algebraic number, Acta Arith. 149 (2011), no. 4, 337--346.
\end{thebibliography}
\end{document}